\documentclass[10pt]{amsart}
\usepackage{amsfonts,amssymb,amscd,amsmath,enumerate,verbatim,calc}
\numberwithin{equation}{section}
\newtheorem{thm}{Theorem}[section]
\newtheorem{cor}[thm]{Corollary}
\newtheorem{lem}[thm]{Lemma}
\newtheorem{prop}[thm]{Proposition}
\newtheorem{defn}[thm]{Definition}

\newtheorem{exam}[thm]{Example}
\newtheorem{rem}[thm]{Remark}

\DeclareMathOperator{\Ext}{Ext} \DeclareMathOperator{\Supp}{Supp}
\DeclareMathOperator{\V}{V} \DeclareMathOperator{\Hom}{Hom}
\DeclareMathOperator{\Ker}{Ker} \DeclareMathOperator{\Coker}{Coker}
\DeclareMathOperator{\Image}{Im} 
\DeclareMathOperator{\cd}{cd} \DeclareMathOperator{\q}{q}

\DeclareMathOperator{\Min}{Min} \DeclareMathOperator{\Max}{Max}
\DeclareMathOperator{\lc}{H} 
 
\DeclareMathOperator{\G}{\Gamma} 
 \DeclareMathOperator{\h}{H}

\DeclareMathOperator{\Ass}{Ass}

\newcommand{\fa}{\mathfrak{a}}
\newcommand{\fb}{\mathfrak{b}}

\newcommand{\fm}{\mathfrak{m}}
\newcommand{\fp}{\mathfrak{p}}

\newcommand{\lo}{\longrightarrow}

\begin{document}

\title[Upper bounds, cofiniteness, and artinianness]
 {Upper bounds, cofiniteness, and artinianness of local
 cohomology modules defined by a pair of ideals}

\bibliographystyle{amsplain}

   \author[Moharram Aghapournahr]{M. Aghapournahr$^{1}$}
\address{$^{1}$ Department of Mathematics, Faculty of Science, Arak University,
Arak, 38156-8-8349, Iran.}
\email{m-aghapour@araku.ac.ir}

   \author[Kh. Ahmadi-amoli]{Kh. Ahmadi-amoli$^{2}$}
\address{$^{2}$ Department of Mathematics, Payame Noor University, Tehran,
 19395-3697, Iran.}
\email{khahmadi@pnu.ac.ir}

    \author[M. Y. Sadeghi]{M. Y. Sadeghi$^{3}$}
\address{$^{3}$ Department of Mathematics, Payame Noor University, Tehran,
 19395-3697, Iran.}
\email{m.sadeghi@phd.pnu.ac.ir}




\begin{abstract}
Let $R$ be a commutative noetherian ring,
$I,J$ be two ideals of $R$, $M$ be an $R$-module,
and $\mathcal{S}$ be a Serre class of $R$-modules.
A positive answer to the Huneke$^,$s conjecture is
given for a noetherian ring $R$ and minimax
$R$-module $M$ of krull dimension less than 3,
with respect to $\mathcal{S}$.
There are some results on cofiniteness and artinianness of
local cohomology modules with respect to a pair of ideals.
For a ZD-module $M$ of finite krull dimension and an integer
$n\in\mathbb{N}$, if $\lc^{i}_{I,J}(M)\in\mathcal{S}$ for all
$i>n$, then
$\lc^{i}_{I,J}(M)/\fa^{j}\lc^{i}_{I,J}(M)\in\mathcal{S}$
for any $\fa\in\tilde{W}(I,J)$, all $i\geq n$,
and all $j\geq0$.
By introducing the concept of Seree cohomological
dimension of $M$ with respect to $(I,J)$,
for an integer $r\in\mathbb{N}_0$,
$\lc^{j}_{I,J}(R)\in\mathcal{S}$
for all $j>r$ iff $\lc^{j}_{I,J}(M)\in\mathcal{S}$
for all $j>r$ and any finite $R$-module $M$.
\end{abstract}

\maketitle


\section{Introduction}


As a generalization of the notion of local cohomology modules,
R. Takahashi, Y. Yoshino, and T. Yoshizawa \cite{TakYoYo},
introduced local cohomology modules with respect to a pair
of ideals. This paper is concerned about this new notion of
local cohomology and Serre subcategories. For notations and
terminologies not given in this paper, if necessary,
the reader is referred to \cite{TakYoYo} and \cite{AghMel}.
\\Throughout this paper, $R$ is denoted a commutative
noetherian ring with non-zero identity, $I$ , $J$ are denoted
two ideals of $R$, and $M$ is denoted an arbitrary $R$-module.
The $(I,J)$-torsion submodule $\Gamma_{I,J}(M)$ of $M$ is a
submodule of $M$ consists of all elements $x$ of $M$ with
Supp$(Rx)\subseteq W(I,J)$, in which
$$W(I,J)=\{~\fp\in\textmd{Spec}(R)\mid I^n\subseteq
\fp+J \textmd{\ \ for an integer} \ n\geq1\}.$$
For an integer $n$, the $n$-th local cohomology functor
$\lc^{n}_{I,J}$ with respect to $(I,J)$ is the $n$-th
right derived functor of $\Gamma_{I,J}$.
The $R$-module $\lc^{n}_{I,J}(M)$ is called the $n$-th
local cohomology module of $M$ with respect to $(I,J)$.
In the case $J=0$, $\lc^{n}_{I,J}$ coincides with
the ordinary functor $\lc^{n}_{I}$.
Also, we are concerned with the following set of ideals of $R$:
$$\tilde{W}(I,J)=\{~\fa\trianglelefteq R\mid
I^n\subseteq \fa+J \textmd{\ \ for an integer} \ n\geq0\}.$$
A class of $R$-modules is called a Serre subcategory
(or Serre class) of the category of $R$-modules when it is
closed under taking submodules, quotients and extensions.
Always, $\mathcal{S}$ stands for a Serre class.
\\According to the third Huneke$^,$s problem on local cohomology
\cite[Conjecture 4.3]{Hu}, one of the main problem in commutative
algebra is finiteness of the socle of local cohomology modules
on a local ring. Solving this problem, gives an answer to the
finiteness of the set of associated primes of local cohomology modules.
\\On this area, some remarkable attempts have been done, e.g. see
\cite{HuSh}, \cite{Ly}, \cite{Ly1}, \cite{MarVa}, and
\cite{AsgTo}.\\
In section 2, we give a positive answer to the Huneke$^,$s
conjecture more general for an arbitrary Serre subcategory
$\mathcal{S}$, instead of the category of finitely generated
modules. Let $R$ be a noetherian (not necessary local) ring.
Let $R/\fm\in\mathcal{S}$ for all $\fm\in\textmd{Max}(R)$.
For any minimax $R$-module $M$ of krull dimension less than 3,
we show that
$\Ext^{j}_R\big(R/\fm,\lc^{i}_{I}(M)\big)\in\mathcal{S}$
for any $\fm\in\textmd{Max}(R)\cap\V(I)$ and all $i,j\geq0$.
In particular $\Hom_R\big(R/\fm,\lc^{i}_{I}(M)\big)\in\mathcal{S}$
for any $\fm\in\textmd{Max}(R)\cap\V(I)$ and all $i\geq0$
(see Theorem \ref{2.14}).\\
We get the same result for local cohomology modules with
respect to a pair of ideals, but in local case
(see Theorem \ref{2.18}).\\
In section 3, we obtain some results on cofiniteness
and artinianness of local cohomology with respect to a
pair of ideals.\\
M. Aghapournahr and L. Melkersson in \cite{AghMel},
obtained some conditions in which the ordinary local
cohomology $\lc^{i}_{I}(M)$ belongs to $\mathcal{S}$
for all $i<n$ (from below).
In section 4, as a complement of this work, we study some
conditions in which the local cohomology $\lc^{i}_{I,J}(M)$
belongs to $\mathcal{S}$ for all $i>n$ (from top).
For a ZD-module $M$ of finite krull dimension, we show that if
the integer $n\in\mathbb{N}$ is such that
$\lc^{i}_{I,J}(M)\in\mathcal{S}$ for all $i>n$, then the modules
$\lc^{i}_{I,J}(M)/\fa^{j}\lc^{i}_{I,J}(M)\in\mathcal{S}$ for any $\fa\in\tilde{W}(I,J)$, all $i\geq n$, and all $j\geq0$
(see Theorem \ref{4.4}).
Replacing $\mathcal{S}$ with some familiar
Serre subcategories such as zero modules, finite modules,
and artinian modules (resp.), we show that a necessary
and sufficient condition for $\lc^{n}_{I,J}(M)$ to be zero,
inite, and artinian (resp.), is the existence of an integer
$m\in\mathbb{N}_0$ such that $I^m\lc^{n}_{I,J}(M)$ is zero,
finite, and artinian (resp.) (see Corollary \ref{4.6}).\\
More generally, for a finite $R$-module $M$ if
$n\in\mathbb{N}_0$ is such that $\lc^{i}_{I,J}(M)$ belongs to
$\mathcal{S}$ for all $i>n$ and $\fb$ is an ideal of $R$
such that $\lc^{n}_{I,J}(M/\fb M)$ belongs to $\mathcal{S}$,
then $\lc^{n}_{I,J}(M)/{\fb}\lc^{n}_{I,J}(M)$ belongs to
$\mathcal S$ (see Theorem \ref{4.7}).
As a consequence of this theorem, we obtain a similar
result as Corollary \ref{4.6}, for the finite module $M$
(see Corollary \ref{4.10}).
Other consequence of this theorem is concerned about
finiteness of $\lc^{n}_{I,J}(M)$ where $n=\textmd{cd}(I,J,M)$
or $n=\textmd{dim}M\geq1$ (see Corollary \ref{4.11}).\\
Finally, these results motivate us to introduce
the concept of Serre cohomological dimension of $M$ with
respect to $(I,J)$, $\textmd{cd}_{\mathcal{S}}(I,J,M)$, as the
supremum of non-negative integers $i$ such that
$\lc^{i}_{I,J}(M)\not\in\mathcal{S}$. We give some
characterizations to $\textmd{cd}_{\mathcal{S}}(I,J,M)$
(Corollary \ref{4.20}) and as a main result of its properties,
we show that for an integer $r\in\mathbb{N}_0$,
$\lc^{j}_{I,J}(R)\in\mathcal{S}$ for all $j>r$ iff
$\lc^{j}_{I,J}(M)\in\mathcal{S}$ for all $j>r$
and any finite $R$-module $M$
(see Corolary \ref{4.23})

\maketitle



\section{The membership of the
$\textmd{Hom}_R\big(-,\lc^{t}_{I,J}(M)\big)$ in Serre classes}


In this section, Proposition \ref{2.5} plays a main
role to obtain our results. For this purpose we need
the following Lemmas.

\begin{lem}\label{2.1}
For a Serre class $\mathcal{S}$, we have $\mathcal{S}\neq{0}$ if and only if
$R/\fm\in\mathcal{S}$ for some $\fm\in\Max(R)$.

\begin{proof}
$(\Rightarrow)$ Let $L\in\mathcal{S}$ and $0\neq x\in L$. Then $(0:_Rx)\subseteq\fm$
for some $\fm\in$ Max$(R)$. Now, since $Rx\in\mathcal{S}$, the assertion follows from
the natural epimorphism $Rx\cong R/(0:_Rx)\rightarrow R/\fm$.\\
$(\Leftarrow)$ It is obvious.
\end{proof}
\end{lem}


\begin{lem}\label{2.2}
Let $\mathcal{FL}$ be the class of finite length $R$-modules. Then
$\mathcal{FL}\subseteq\mathcal{S}$ if and only if $R/\fm\in\mathcal{S}$ for all
$\fm\in\Max(R)$.
\end{lem}
\begin{proof}
$(\Rightarrow)$ It is obvious.\\
$(\Leftarrow)$ Let $N\in\mathcal{FL}$ and $\ell:=\ell_R(N)$. So, consider the
following chain of $R$-submodules of $N$:
$$0=N_0\subseteq N_1\subseteq\cdots\subseteq N_\ell=N$$
in which, for all $1\leq j\leq \ell$, $N_j/N_{j-1}\cong R/\fm$  for some
$\fm\in$ Max$(R)$. Now, the assertion can be followed by induction on $\ell$.
\end{proof}


\begin{cor}\label{2.3}
Let $(R,\fm)$ be a local ring and $\mathcal{S}\neq0$. Then
$\mathcal{FL}\subseteq\mathcal{S}$.
\end{cor}


\begin{exam}\label{2.4}
\emph{In general case, it not true that $R/\fm\in\mathcal{S}$
for all $\fm\in\Max(R)$. To see this, let $R$ be a non-local
ring and $\fm\in$ Max$(R)$.
Let  $I$ be an ideal of $R$ such that $I\nsubseteq\fm$.
Let $\mathcal{S}$ be the
class of $I$-cofinite minimax $R$-modules,
(see \cite[Corollary 4.4]{Mel1}).
Then Supp$(R/\fm)\nsubseteq V(I)$ and so
$R/\fm\not\in\mathcal{S}$.
For example,
let $R:=\mathbb{Z}[x]$, $\fm:=(x-1)R$, and $I:=xR$.}
\end{exam}



\begin{prop}\label{2.5}
For a noetherian ring $R$, we have
\begin{enumerate}

  \item[\rm{(}i\rm{)}] If $R/\fm\in\mathcal{S}$, for any $\fm\in$ Max$(R)$, and
  $M$ is a finite or an artinian $R$-module, then $\Ext^j_R(R/\fm,M)\in\mathcal{S}$
  for any $\fm\in$ Max$(R)$ and all $j\geq0$.

  \item[\rm{(}ii\rm{)}] If $R/\fm\in\mathcal{S}$, for any $\fm\in$ Max$(R)$, and $M$
  is a minimax $R$-module, then $\Ext^j_R(R/\fm,M)\in\mathcal{S}$
  for any $\fm\in$ Max$(R)$ and all $j\geq0$.

  \item[\rm{(}iii\rm{)}] If $(R,\fm)$ be a local ring, $\mathcal{S}\neq0$, and $M$
  be a minimax $R$-module, then $\Ext^j_R(R/\fm,M)\in\mathcal{S}$ for all $j\geq0$.

\end{enumerate}

\begin{proof}
(i) Let $\fm\in$ Max$(R)$ and $j\geq0$. Since $\Ext^j_R(R/\fm,M)$ is finite and is
annihilated by $\fm$, hence $\Ext^j_R(R/\fm,M)$ has finite length. Now, the result
follows from Lemma \ref {2.2}.\\
(ii) Since $M$ is minimax, there exists a short exact sequence
$$0\rightarrow N\rightarrow M\rightarrow A\rightarrow0,$$
where $N$ is a finite module and $A$ is an artinian module.
This induces the exact sequence
$$\cdots\rightarrow\Ext^j_R(R/\fm,N)\rightarrow\Ext^j_R(R/\fm,M)\rightarrow
\Ext^j_R(R/\fm,A)\rightarrow\cdots .$$
Now, the assertion follows from part (i).\\
(iii) Apply Lemma \ref {2.1} and part (ii).
\end{proof}
\end{prop}


\begin{cor}\label{2.6}
Let $R/\fm\in\mathcal{S}$ for all $\fm\in$ Max$(R)$. Let $t\in\mathbb{N}_0$
be such that $\lc^{t}_{I,J}(M)$ is a minimax $R$-module. Then
$\Ext^j_R\big(R/\fm,\lc^{t}_{I,J}(M)\big)\in\mathcal{S}$ for all $i\geq0$.
\end{cor}


\begin{lem}\label{2.7}
Let $\fa\in\tilde{W}(I,J)$. Let $X$ be an $R$-module. Then
  $$\big( 0 :_X\fa\big) = \big(0:_{\Gamma_{\fa}(X)}\fa\big) =
  \big(0:_{\Gamma_{\fa,J}(X)}\fa\big) = \big(0:_{\Gamma_{I,J}(X)}\fa\big),$$
      in particular, for any ideal $\fb$ of $R$ with $\fb\supseteq\ I$, we have
 \begin{enumerate}

  \item[\rm{(}i\rm{)}] $\big( 0 :_X\fb\big) = \big(0:_{\Gamma_{\fb}(X)}\fb\big) =
  \big(0:_{\Gamma_{\fb,J}(X)}\fb\big) = \big(0:_{\Gamma_{I,J}(X)}\fb\big) =
  \big(0:_{\Gamma_{I}(X)}\fb\big)$.

  \item[\rm{(}ii\rm{)}] $\big(0:_X\fb\big) = \big(0:_{\Gamma_{\fb,J}(X)}\fb\big)
  \subseteq \big(0:_{\Gamma_{\fb,J}(X)}I\big) \subseteq
  \big(0:_{\Gamma_{I,J}(X)}I\big) = \big( 0 :_{X}I\big)$.

\end{enumerate}

\begin{proof}
All proofs are easy and we leave them to the reader.
\end{proof}
\end{lem}


\begin{prop}\label{2.8}
Let $\fa\in\tilde{W}(I,J)$ and $t\in\mathbb{N}_0$. Consider the natural
homomorphism
$$\psi : \Ext^t_R\big(R/\fa,M\big)\longrightarrow \Hom_R\big(R/\fa,
\lc^{t}_{I,J}(M)\big).$$

\begin{enumerate}

  \item[\rm{(}i\rm{)}] If $\Ext^{t-j}_R\big(R/\fa,\lc^{j}_{I,J}(M)\big)
  \in\mathcal{S}$ for all $j<t$, then \emph{Ker} $\psi\in\mathcal{S}$.

  \item[\rm{(}ii\rm{)}] If $\Ext^{t+1-j}_R\big(R/\fa,\lc^{j}_{I,J}(M)\big)
  \in\mathcal{S}$ for all $j<t$, then \emph{Coker} $\psi\in\mathcal{S}$.

  \item[\rm{(}iii\rm{)}] If $\Ext^{n-j}_R\big(R/\fa,\lc^{j}_{I,J}(M)\big)
  \in\mathcal{S}$ for $n=t,\ t+1$ and for all $j<t$, then \emph{Ker} $\psi$
  and \emph{Coker} $\psi$ both belong to $\mathcal{S}$.
  Thus $\Ext^t_R\big(R/\fa,M\big)\in\mathcal{S}$ iff
  $\Hom_R\big(R/\fa,\lc^{t}_{I,J}(M)\big)\in\mathcal{S}$.

\end{enumerate}
\begin{proof}
Let $\textmd{F} (-) = \Hom_R\big(R/\fa,-\big)$ and
$\textmd{G} (-) = {\Gamma_{I,J}\big(-\big)}$.
By Lemma \ref {2.7}, $\textmd{FG} =\textmd{F}$.
Now, the result can be followed by \cite[Proposition 3.1]{AghMel1}.
\end{proof}
\end{prop}


The next result can be a generalization of main results of
\cite{AsgTo}, \cite{KhaSal}, \cite{BroLash},
\cite{DivMaf}, \cite{LoSahYas}, \cite{BahNag},
and \cite{AsaKhaSal}.

\begin{thm} \label {2.9}
Let $\fa\in\tilde{W}(I,J)$ and $~t\in\mathbb{N}_0$ be such that
$\Ext^{t}_R\big(R/\fa,M\big)\in\mathcal{S}$
and $\Ext^{n-j}_R\big(R/\fa,\lc^{j}_{I,J}(M)\big)$ $\in\mathcal{S}$ for $n=t,\ t+1$
and for all $j<t$. Then for any submodule $N$ of $\lc^{t}_{I,J}(M)$ such that
$\Ext^1_R\big(R/\fa,N\big)\in\mathcal{S}$, we have

\begin{enumerate}

  \item[\rm{(}i\rm{)}] $\Hom_R\big(R/\fa,\lc^{t}_{I,J}(M)/N\big)\in\mathcal{S}$.

  \item[\rm{(}ii\rm{)}] $\Hom_R\big(L,\lc^{t}_{I,J}(M)/N\big)\in\mathcal{S}$
  for any finite $R$-module $L$ with $\Supp(L)\subseteq\textmd{V}(\fa)$.

  \item[\rm{(}iii\rm{)}] $\Hom_R\big(R/\fp,\lc^{t}_{I,J}(M)/N\big)\in\mathcal{S}$
  for any $\fp\in\textmd{V}(\fa)$.
\end{enumerate}
All the statements are hold for $\fa=I$.
\begin{proof}
(i) Apply Proposition \ref{2.8} and the exact sequence
$$0\rightarrow N\rightarrow\lc^{t}_{I,J}(M)\rightarrow\lc^{t}_{I,J}(M)/N\rightarrow0.$$
For (ii) and (iii) apply \cite[Theorem 2.10]{AhSad}.
\end{proof}
\end{thm}


\begin{cor}\label{2.10}
Let $(R,\fm)$ be a local ring and $\fa\in\tilde{W}(I,J)$.
Let $t\in\mathbb{N}_0$ be such that
$\Ext^{t}_R(R/\fa,M)\in\mathcal{S}$ and
$\Ext^{n-j}_R\big(R/\fa,\lc^{j}_{I,J}(M)\big)\in\mathcal{S}$
for $n=t,\ t+1$ and all $j<t$.
Then for any submodule $N$ of $\lc^{t}_{I,J}(M)$ such that
$\Ext^1_R\big(R/\fa,N\big)\in\mathcal{S}$, we have
$\Hom_R\big(R/\fa,\lc^{t}_{I,J}(M)/N\big)\in\mathcal{S}$,
specially $\Hom_R\big(R/\fm,\lc^{t}_{I,J}(M)/N\big)\in\mathcal{S}$.
\end{cor}


\begin{exam} \label{2.11}
\emph{In Theorem \ref{2.9},
the assumption $\Ext^{t}_R\big(R/\fa,M\big)\in\mathcal{S}$ is necessary.
To see this, let $(R,\fm)$ be a local Gorenstein ring of positive dimension $d$,
and $\mathcal{S}=0$. Then $\lc^{i}_{\fm}(R)=0$ for $i<d$. But we have
$\Ext^{d}_R\big(R/\fm,R\big)\cong\Hom_R\big(R/\fm,\lc^{d}_{\fm}(R)\big)\cong
\Hom_R\big(R/\fm,\emph{E}(R/\fm)\big)\cong R/\fm\neq0$}.
\end{exam}


One of the main result of this paper is the following theorem which is a
generalization of \cite[Theorem 2.12]{AsgTo}

\begin{thm} \label {2.12}
Let $R/\fm\in\mathcal{S}$ for all $\fm\in$ $\emph{Max}(R)$. Let $M$ be a minimax
$R$-module and $t\in\mathbb{N}_0$ be such that
$\Ext^{n-j}_R\big(R/\fm,\lc^{j}_{I,J}(M)\big)$ $\in\mathcal{S}$ for
$n=t,\ t+1$, all $j<t$, and all $\fm\in$ $\emph{Max}(R)$.
Then for any submodule $N$ of $\lc^{t}_{I,J}(M)$ such that
$\Ext^1_R\big(R/\fm,N\big)\in\mathcal{S}$, we have
$\Hom_R\big(R/\fm,\lc^{t}_{I,J}(M)/N\big)\in\mathcal{S}$,
for any $\fm\in\emph{W}_\emph{Max}(I,J) := \emph{Max}(R)\cap\emph{W}(I,J)$.
\end{thm}
\begin{proof}
The assertion follows from Proposition \ref{2.5} (ii) and Theorem \ref{2.9}.
\end{proof}


\begin{cor}\label{2.13}
Let $(R,\fm)$ be a local ring, $\mathcal{S}\neq0$ and $M$ be a minimax $R$-module.
Let $~t\in\mathbb{N}_0$ be such that $\Ext^{n-j}_R\big(R/\fm,\lc^{j}_{I,J}(M)\big)$
$\in\mathcal{S}$ for $n=t,\ t+1$ and for all $j<t$.
Then for any minimax submodule $N$ of $\lc^{t}_{I,J}(M)$,
we have $\Hom_R\big(R/\fm,\lc^{t}_{I,J}(M)/N\big)\in\mathcal{S}.$
\end{cor}
\begin{proof}
Apply Proposition \ref{2.5} and Theorem \ref{2.12}.
\end{proof}


The following familiar conjecture is due to Huneke \cite{Hu}.\\
\textbf{Conjecture}. Let $(R,\fm,k)$ be a regular local ring and $I$ be an ideal of $R$.
For all $n$, soc$\big(\lc^{n}_{I}(R)\big)$ is finitely generated.\\
As we mentioned in the introduction, the following theorem which is one of
the main result of this section, can be a positive answer to Huneke$^,$s
conjecture. In fact the following theorem proves a generalization of the conjecture
for an arbitrary noetherian ring $R$ (not necessary regular local one), a minimax
$R$-module $M$ of krull dimension less than 3, and a Serre class $\mathcal{S}$.

\begin{thm} \label {2.14}
Let $R$ be a noetherian ring and $M$ be a minimax $R$-module
of krull dimension less than \emph{3}.
Let $R/\fm\in\mathcal{S}$ for any $\fm\in\emph{Max}(R)$.
Then $\Ext^{j}_R\big(R/\fm,\lc^{i}_{I}(M)\big)\in\mathcal{S}$ for any $\fm\in\emph{Max}(R)\cap\V(I)$ and all $i,j\geq0$.
In particular $\Hom_R\big(R/\fm,\lc^{i}_{I}(M)\big)\in\mathcal{S}$
for any $\fm\in\emph{Max}(R)\cap\V(I)$ and all $i\geq0$.
\end{thm}

\begin{proof}
Let $\fm\in\textmd{Max}(R)\cap\V(I)$. By Proposition \ref{2.5} (ii)
and the Grothendieck$^,$s vanishing theorem there is
nothing to prove for cases $i=0$ and $i>2$.
Now, assume that $0<i\leq2$. If dim $M=2$ and $i=2$, then
the result follows from \cite[Corollary 3.3]{DivEsm}
and Proposition \ref{2.5} (i).
Also, in the case $i=1$, the result is obtained from
\cite[Theorem 2.3 ]{AghTahVah}, by replacing $s:=j$,
$t:=1$, and $N:=R/\fm$, and Proposition \ref{2.5} (ii).
Finally if dim$M\leq1$, we can obtain the desired
result in similar way.
\end{proof}


\begin{cor} \label {2.15}
Let $R$ be a noetherian ring of dimension less than \emph{3}
and let $M$ be a minimax $R$-module. Let $R/\fm\in\mathcal{S}$
for any $\fm\in\emph{Max}(R)$. Then
$\Ext^{j}_R\big(R/\fm,\lc^{i}_{I}(M)\big)\in\mathcal{S}$ for any
$\fm\in\emph{Max}(R)\cap\V(I)$ and all $i,j\geq0$. In particular for the class of finite $R$-modules.
\end{cor}


\begin{cor}\label{2.16}
Let $R$ be a noetherian ring. Let $M$ be a minimax $R$-module
of krull dimension less than \emph{3}. Then the Bass numbers of
$\lc^{i}_{I}(M)$ are finite for all $i\geq0$, in particular it is true
when $\dim R\leqslant 2$.
\end{cor}


One can generalize Theorem \ref{2.14} for local cohomology
modules with respect to a pair of ideals, but in local case.
To do this, we need the
following lemma.

\begin{lem}\label{2.17}
Let $(R,\fm)$ be a local ring and $M$ be a minimax $R$-module
of finite krull dimension d. Then $\lc^{d}_{I,J}(M)$ is artinian.
\begin{proof}
Let $N$ be a finite submodule of $M$ such that $A:=M/N$
is artinian. Since $\textmd{dim}~N\leq d$, by \cite[Theorem 2.1]{ChWa},
$\lc^{d}_{I,J}(N)$ is artinian. Also, the exact sequence
$0\rightarrow N\rightarrow M \rightarrow A\rightarrow0$
induces the following exact sequence
$$0\rightarrow\Gamma_{I,J}(N)\rightarrow\Gamma_{I,J}(M)\rightarrow \Gamma_{I,J}(A)
\rightarrow \lc^{1}_{I,J}(N)\rightarrow \lc^{1}_{I,J}(M) \rightarrow0,$$
and $\lc^{i}_{I,J}(N)\cong\lc^{i}_{I,J}(M)$ for all $i\geq2$.
By the exactness of the above sequence, it is easy to see that $\lc^{d}_{I,J}(M)$ is
artinian.
\end{proof}
\end{lem}


\begin{thm}\label{2.18}
Let $(R,\fm)$ be a local ring and $M$ be a minimax $R$-module
of krull dimension less than \emph{3}. Let $\mathcal{S}\neq0$.
Then $\Ext^{j}_R\big(R/\fm,\lc^{i}_{I,J}(M)\big)\in\mathcal{S}$
for all $i,j\geq0$.
\begin{proof}
Apply Lemma \ref{2.17}, Proposition \ref{2.5}, and
the same method of the proof of Theorem \ref{2.14}.
\end{proof}
\end{thm}


The next ressults can be useful for finiteness of Bass numbers of
local cohomology modules over a noetherian ring and modules
of krull dimension 3.
Recall that an $R$-module $M$ is called locally minimax if
$M_\fm$ is minimax for any $\fm\in$ Max$(R)$ (see \cite{AghMel2}).

\begin{thm} \label {2.19}
Let $R$ be a noetherian ring and $M$ be an $R$-module
of krull dimension \emph{3}.
Let $R/\fm\in\mathcal{S}$ for any $\fm\in\emph{Max}(R)$.
Then $\Ext^{j}_R\big(R/\fm,\lc^{i}_{I}(M)\big)\in\mathcal{S}$
for any $\fm\in\emph{Max}(R)\cap V(I)$ and all $i,j\geq0$,
if one of the following conditions holds:
\begin{enumerate}

  \item[\rm{(}i\rm{)}] $M$ and $\lc^{1}_{I}(M)$ are minimax;

  \item[\rm{(}ii\rm{)}] $M$ is finite and $\lc^{1}_{I}(M)$ is
  locally minimax;

  \item[\rm{(}iii\rm{)}] $M$ and $I^{m}\lc^{2}_{I}(M)$ are minimax
  for some $m\in\mathbb{N}_0$;

  \item[\rm{(}iv\rm{)}] $M$ is minimax, $I^{m}\lc^{2}_{I}(M)$ is
  locally minimax, and $\Hom_{R}\big(R/I,I^{m}\lc^{2}_{I}(M)\big)$ is
  finite for some $m\in\mathbb{N}_0$;

  \item[\rm{(}v\rm{)}] $(R,\fm)$ is local, $M$ is minimax, and $I^{m}\lc^{2}_{I}(M)$ is locally minimax for some $m\in\mathbb{N}_0$.
\end{enumerate}

\begin{proof}
Let $\fm\in\emph{Max}(R)\cap\V(I)$.

(i) By \cite[Corollary 3.3]{DivEsm}, $\lc^{3}_{I}(M)$ is artinian.
Thus the assertion is true for $i=0,1,3$, by Proposition \ref{2.5}.
Also, for $i=2$, we apply \cite[Theorem 2.3]{AghTahVah}, by replacing
$s:=j$, $t:=2$, and $N:=R/\fm$.

(ii) By \cite[Theorem 2.3]{BahNag} and \cite[Theorem 2.6]{AghMel2},
$\lc^{1}_{I}(M)$ is minimax. Now, the assertion follows from part (i).

(iii) By the short exact sequence
$$0\rightarrow I^m\lc^{2}_{I}(M)\rightarrow
\lc^{2}_{I}(M)\rightarrow \lc^{2}_{I}(M)/I^m\lc^{2}_{I}(M)
\rightarrow0,$$
and \cite[Theorem 3.1]{AsgTo}, for the class of minimax $R$-modules,
we get $\lc^{2}_{I}(M)$ is minimax. Now, for $i=0,2,3$, the claim
follows from \cite[Corollary 3.3]{DivEsm} and Proposition \ref{2.5}.
In the case $i=2$, apply \cite[Theorem 2.3 ]{AghTahVah},
by replacing $s:=j$, $t:=1$, $N:=R/\fm$ and Proposition \ref{2.5}.

(iv) Apply \cite[Theorem 2.6]{AghMel2} and part (iii).

(v) By proof of Lemma \ref{2.17}, we may assume that $M$
is finite. Now, the assertion follows from
\cite[Propositions 3.4 , 2.2]{AghMel2} and part (iii).
\end{proof}
\end{thm}


\begin{cor}\label{2.20}
Let $R$ be a noetherian ring and $M$ be a minimax $R$-module
of krull dimension $d\leq\emph{3}$. Then the Bass numbers of $\lc^{i}_{I}(M)$
are finite for all $i\geq0$, if one of the following conditions holds:
\begin{enumerate}

  \item[\rm{(}i\rm{)}] $M$ and $\lc^{1}_{I}(M)$ are minimax;

  \item[\rm{(}ii\rm{)}] $M$ is finite and $\lc^{1}_{I}(M)$ is
  locally minimax;

  \item[\rm{(}iii\rm{)}] $M$ and $I^{m}\lc^{2}_{I}(M)$ are minimax
  for some $m\in\mathbb{N}_0$;

  \item[\rm{(}iv\rm{)}] $M$ is minimax, $I^{m}\lc^{2}_{I}(M)$ is
  locally minimax, and $\Hom_{R}\big(R/I,I^{m}\lc^{2}_{I}(M)\big)$ is
  finite for some $m\in\mathbb{N}_0$;

  \item[\rm{(}v\rm{)}] $M$ is minimax, and $I^{m}\lc^{2}_{I}(M)$ is locally minimax, for some $m\in\mathbb{N}_0$.
\end{enumerate}
In particular the statments are true when $\dim R\leqslant 3$.
\end{cor}
\begin{proof}
For part (v), apply localization and Theorem \ref{2.19} (v).
\end{proof}


\section{Cofiniteness and Artinianness of $\lc^{i}_{I,J}(M)$}


In this section, we need the concept of $(I,J)$-cofinite modules and $\cd(I,J,M)$.
An $R$-module $M$ is called $(I,J)$-cofinite if Supp$(M)\subseteq W(I,J)$
and $\Ext^i_R(R/I,M)$ is a finite $R$-module, for all $i\geq0$.
Also $~\cd(I,J,M)={\sup\big\{i\in\mathbb{N}_0}\mid H^i_{I,J}(M)\neq0\big\}$
(see \cite{TehTa} and \cite{ChWa}, resp.)

\begin{thm}\label{3.1}
Let $(R,\fm)$ be a local ring and $M$ be a finite $R$-module.
Let $t\in\mathbb{N}$ be an integer such that
$\Supp(\lc^{i}_{I,J}(M))\subseteq \{ \fm \}$ for all $i<t$.
Then, for all $i< t$, the $R$-module $\lc^{i}_{I,J}(M)$
 is artinian and $I$-cofinite.
\begin{proof}
We do this by induction on $t$. When $t=1$, it is obvious that
$\lc^{0}_{I,J}(M)$ is artinian $I$-cofinite $R$-module,
since it is a finite module with support in $\{\fm\}$. Now, suppose that
$t\geq2$ and the case $t-1$ is settled.
By \cite[Corollary 1.13]{TakYoYo}, we may assume that $M$ is
$(I,J)$-torsion free, and so $I$-torsion free $R$-module.
Therefore, by \cite[Lemma 2.1.1]{BroSh}, there exists
$x\in I\setminus\bigcup_{\fp\in \Ass(M)}\fp$ such that
$0\rightarrow M^{{.x}\atop{\longrightarrow}} M\rightarrow M/xM \rightarrow0$
is exact. Now, by the exact sequence
$$\lc^{t-2}_{I,J}(M)\rightarrow\lc^{t-2}_{I,J}(M/xM)\rightarrow
\lc^{t-1}_{I,J}(M)^{{.x}\atop{\lo}}\lc^{t-1}_{I,J}(M),$$
we get the following exact sequence
$$\lc^{t-2}_{I,J}(M)\rightarrow\lc^{t-2}_{I,J}(M/xM)\rightarrow
\big(0:_{\lc^{t-1}_{I,J}(M)}x\big)\rightarrow0.$$
Thus by inductive hypothesis $\lc^{t-2}_{I,J}(M/xM)$ is
artinian $I$-cofinite . So $\big(0:_{\lc^{t-1}_{I,J}(M)}x\big)$
is artinian $I$-cofinite. As Supp$(\lc^{t-1}_{I,J}(M))\subseteq
 \{ \fm \}$, hence $\lc^{t-1}_{I,J}(M)$ is $I$-torsion.
 Now, the assertion follows from \cite[Proposition 4.1]{Mel1}.
\end{proof}
\end{thm}


\begin{thm} \label {3.2}
Let $M$ be an $R$-module such that $\Ext^{i}_R(R/I,M)$ is finite
for all $i\geq0$. Let $~t\in\mathbb{N}_0$ be such that
$\lc^{i}_{I,J}(M)$ is $(I,J)$-cofinite for all $i\neq t$.
Then $\lc^{t}_{I,J}(M)$ is $(I,J)$-cofinite.
\begin{proof}
Apply \cite[Theorem 3.11]{AhSad}.
\end{proof}
\end{thm}


\begin{thm} \label {3.3}
Let $M$ be a finite $R$-module. If
$~\emph{cd}(I,J,M)\leq1$, then
$~\lc^{i}_{I,J}(M)$ is $(I,J)$-cofinite for all $i\geqslant 0$.
\begin{proof}
When $i=0$ and $i\geq2$, the claim is true, since $\Gamma_{I,J}(M)$
is finite and $\cd(I,J,M)\leq1$.  For $i=1$, apply Theorem \ref{3.2}.
\end{proof}
\end{thm}


\begin{thm} \label {3.4}
Let $(R,\fm)$ be a local ring and $M$ be a finite $R$-module with $\dim M=n$.
Then $\lc^{n}_{I,J}(M)$ is artinian and $I$-cofinite. In fact,
$\Ext^i_R\big(R/I,\lc^{n}_{I,J}(M)\big)$ has finite length for all $i$.
\begin{proof}
The assertion follows from \cite[Theorem 2.1]{ChWa}, \cite[Theorem 2.3]{Ch} and
\cite[Theorem 3]{DelMar}.
\end{proof}
\end{thm}

\begin{thm} \label {3.5}
Let $(R,\fm)$ be a local ring and $M$ be a finite $R$-module.
If $\dim M\leqslant 2$, then $\lc^{i}_{I,J}(M)$ is $(I,J)$-cofinite
for all $i\geqslant 0$, in particular it is true when $\dim R\leqslant 2$.
\begin{proof}
Apply Theorems \ref{3.4} , {3.2}.
\end{proof}
\end{thm}




\section{Upper bounds of $\lc^{i}_{I,J}(M)$}


In this section, we introduce the concept of Serre cohomological dimension of
$M$ with respect to a pair of ideals $(I,J)$, but first, we
characterize the membership of $\lc^{i}_{I,J}(M)$ in a Serre class from
upper bound.


\begin{thm}\label{4.1}
Let $n\in\mathbb{N}_0$ and $M$ be a finite $R$--module.
Then the following statements are equivalent:

\begin{enumerate}
  \item[\rm{(}i\rm{)}] $\lc^{i}_{I,J}(M)$ is in $\mathcal S$ for all $i> n$.

  \item[\rm{(}ii\rm{)}] $\lc^{i}_{I,J}(N)$ is in $\mathcal S$ for all $i> n$
  and for any finite $R$-module $N$ such that $\Supp_R(N)\subseteq\Supp_R(M)$.

  \item[\rm{(}iii\rm{)}] $\lc^{i}_{I,J}(R/\fp)$ is in $\mathcal S$ for all
  $\fp\in\Supp_R(M)$ and for all $i> n$.

  \item[\rm{(}iv\rm{)}] $\lc^{i}_{I,J}(R/\fp)$ is in $\mathcal S$ for all
  $\fp\in\Min\Ass_R(M)$ and for all $i> n$.

\end{enumerate}
\end{thm}
\begin{proof}
Apply the method of the proof of Theorem 3.1 in \cite{AghMel}
to $\lc^{i}_{I,J}$ .
\end{proof}

\begin{cor}\label{4.2}
Let $M$ , $N$ be finite $R$-modules such that $~\Supp(N)\subseteq\Supp(M)$ and
$n\in\mathbb{N}_0$. If $\lc^{i}_{I,J}(M)\in\mathcal S$ for all $i> n$, then
$\lc^{i}_{I,J}(N)\in\mathcal S$ for all $i> n$.
\end{cor}


\begin{lem}\label{4.3}
For any $R$-module $N$, the following statements are fulfilled.

\begin{enumerate}

  \item[\rm{(}i\rm{)}] $\Gamma_{\fa}(N)\subseteq\Gamma_{\fa,J}(N)
  \subseteq\Gamma_{I,J}(N)$ for any $\fa\in\tilde{W}(I,J)$.

  \item[\rm{(}ii\rm{)}] $\Gamma_{I,J}(N)=0$ if and only if $\Gamma_{\fa,J}(N)=0$
  for any $\fa\in\tilde{W}(I,J)$.

  \item[\rm{(}iii\rm{)}] $\Gamma_{I,J}(N)=N$ if and only if there exists
  $\fa\in\tilde{W}(I,J)$ such that $\Gamma_{\fa,J}(N)=N$.

  \item[\rm{(}iv\rm{)}] If there exists $\fa\in\tilde{W}(I,J)$ such that
  $\Gamma_{\fa}(N)=N$, then $\Gamma_{I,J}(N)=N$.

  \item[\rm{(}v\rm{)}] $\Gamma_{I,J}(N)=\cup_{\fa\in\tilde{W}(I,J)}\Gamma_{\fa}(N)=
  \cup_{\fa\in\tilde{W}(I,J)}\Gamma_{\fa,J}(N)$

\end{enumerate}

\begin{proof}
All these statements follow easily from the definitions. We will only prove
the statement (v). Since $\Gamma_{\fb}(N)\subseteq\Gamma_{\fb,J}(N)$, for
any ideal $\fb$ of $R$, so by \cite[Theorem 3.2]{TakYoYo} and part(i),
we get\\
$\Gamma_{I,J}(N)\subseteq\cup_{\fa\in\tilde{W}(I,J)}\Gamma_{\fa}(N)\subseteq
\cup_{\fa\in\tilde{W}(I,J)}\Gamma_{\fa,J}(N)\subseteq\Gamma_{I,J}(N).$
\end{proof}
\end{lem}


In \cite{DivEsm}, authors introduced the concept of ZD-modules.
An $R$-module $M$ is said to be a ZD-module (zero-divisor module) if for every
submodule $N$ of $M$, the set of zero divisors of $M/N$ is a union of finitely
many prime ideals in Ass$_R(M/N)$. By \cite[Example 2.2]{DivEsm}, the class of
ZD-modules contains modules with finite support, finitely generated, Laskerian,
weakly Laskerian, linearly compact, Matlis reflexive and minimax $R$-modules.

Now, we are in position to prove two other main results of this paper
(Theorem \ref{4.4} and Theorem \ref{4.7}), which the first one can be
considered as a generalization of \cite[Theorem 3.1]{AsgTo}.

\begin{thm}\label{4.4}
Let $M$ be a ZD-module of finite Krull dimension. Let $n\in\mathbb{N}$ be such that
$\lc^{i}_{I,J}(M)\in\mathcal{S}$ for all $i>n$. Then
$\lc^{i}_{I,J}(M)/\fa^{j}\lc^{i}_{I,J}(M)\in\mathcal{S}$ for any
$\fa\in\tilde{W}(I,J)$, all $i\geq n$, and all $j\geq0$.
\begin{proof}
It is enough to verify the assertion for just $i=n$ and $j=1$.
To do this, we use induction on
$d$:=dim$M$. When $d=0$, the result follows from \cite[Theorem 3.2]{TakYoYo} and
Grothendieck$^,$s Vanishing theorem.\\
Next, we assume that $d>0$ and the claim is true for
all $R$-modules of dimension less than $d$. By \cite[Theorem 1.3 (4)]{TakYoYo}, we have
$\lc^{j}_{I,J}(M)\cong \lc^{j}_{I,J}\big(M/\Gamma_{I,J}(M)\big)$ for all $j>0$.
Also, $M/\Gamma_{I,J}(M)$ has dimension not exceeding $d$, and is an $(I,J)$-torsion-free
$R$-module. Therefore we may assume that $\Gamma_{I,J}(M)=0$ and so, by Lemma \ref{4.3},
$\Gamma_{\fa}(M)=0$. By \cite[Lemma 2.4]{DivEsm}, $x\in \fa\setminus\bigcup_{\fp\in \textmd{Ass}(M)}\fp$.
Now, the $R$-module $M/xM$ is ZD-module of dimension $d-1$.
Considering the exact sequence
$$0\rightarrow M^{{.x}\atop{\longrightarrow}} M\rightarrow M/xM \rightarrow0$$
induces a long exact sequence of local cohomology modules, which shows that
$\lc^{i}_{I,J}(M/xM)\in\mathcal{S}$ for all $i>n$.
Thus by inductive hypothesis $\lc^{n}_{I,J}(M/xM)/\fa\lc^{n}_{I,J}(M/xM)\in\mathcal{S}$.
Now, the exact sequence $$\lc^{n}_{I,J}(M)^{{.x}\atop{\longrightarrow}}
\lc^{n}_{I,J}(M)^{{\alpha}\atop{\longrightarrow}} \lc^{n}_{I,J}(M/xM)^{{\beta}\atop
{\longrightarrow}}\lc^{n+1}_{I,J}(M),$$ induces the following exact sequences
$$\lc^{n}_{I,J}(M)^{{.x}\atop{\longrightarrow}}\lc^{n}_{I,J}(M)\rightarrow
N:=\Image\alpha\rightarrow0,$$
$$0\rightarrow N\rightarrow \lc^{n}_{I,J}(M/xM)\rightarrow K:=\Image\beta\rightarrow0.$$
Therefore, the following two sequences

\begin{enumerate}
\item[]$(\ast)$ $~~$ $~~$ $~\lc^{n}_{I,J}(M)/\fa\lc^{n}_{I,J}(M)^{{.x}\atop
{\longrightarrow}}\lc^{n}_{I,J}(M)/\fa\lc^{n}_{I,J}(M)\rightarrow N/\fa N\rightarrow0,$
\end{enumerate}

\begin{enumerate}
\item[]$(\ast\ast)$ $~~$ $~~$ ${\textmd{Tor}^{R}_{1}(R/\fa,K)}\rightarrow N/\fa N
\rightarrow\lc^{n}_{I,J}(M/xM)/\fa\lc^{n}_{I,J}(M/xM)\rightarrow K/\fa K\rightarrow0$
\end{enumerate}
are both exact. Since $x\in \fa$ and from the exact sequence $(\ast)$, we get
$N/\fa N\cong \lc^{n}_{I,J}(M)/\fa\lc^{n}_{I,J}(M)$.
On the other hand, by \cite[Lemma 2.1]{AsgTo}, we have
Tor$^R_1(R/\fa,K)\in\mathcal{S}$. Therefore $N/\fa N\in\mathcal{S}$, by $(\ast\ast)$,
as required.
\end{proof}
\end{thm}


The following result is an immediate consequence of the above theorem.

\begin{cor}\label{4.5}
Let $M$ be a ZD-module of finite Krull dimension. Let $n\in\mathbb{N}$ be such that
$\lc^{i}_{I,J}(M)\in\mathcal{S}$ for all $i>n$. Then $\lc^{n}_{I,J}(M)\in\mathcal{S}$
if and only if there exist $\fa\in\tilde{W}(I,J)$ and $m\in\mathbb{N}_0$ such that
$\fa^m\lc^{n}_{I,J}(M)\in\mathcal{S}$.
\begin{proof}
$(\Rightarrow)$ It is obvious.\\
$(\Leftarrow)$ Apply the short exact sequence $0\rightarrow \fa^m\lc^{n}_{I,J}(M)\rightarrow
\lc^{n}_{I,J}(M)\rightarrow \lc^{n}_{I,J}(M)/\fa^m\lc^{n}_{I,J}(M)\rightarrow0$ and
Theorem \ref{4.3}.
\end{proof}
\end{cor}


Applying Corollary \ref{4.5}, for some familiar Serre classes of modules, we get some
results as follows.

\begin{cor}\label{4.6}
Let $M$ be a ZD-module of finite Krull dimension and let $n\in\mathbb{N}$.
Then the following statements are fulfilled.

\begin{enumerate}
  \item[\rm{(}i\rm{)}] If $\lc^{i}_{I,J}(M)$ is finite for all $i>n$, then
  $\lc^{n}_{I,J}(M)$ is finite if and only if there exist $\fa\in\tilde{W}(I,J)$
  and $m\in\mathbb{N}_0$ such that $\fa^m\lc^{n}_{I,J}(M)$ is finite
  if and only if there exist $\fa\in\tilde{W}(I,J)$
  and $m\in\mathbb{N}_0$ such that
  $\lc^{n}_{I,J}(M)/(0:_{\lc^{n}_{I,J}(M)}\fa^m)$ is finite .

  \item[\rm{(}ii\rm{)}] If $\lc^{i}_{I,J}(M)$ is artinian for all $i>n$, then
  $\lc^{n}_{I,J}(M)$ is artinian if and only if there exist
  $\fa\in\tilde{W}(I,J)$ and $m\in\mathbb{N}_0$
  such that $\fa^m\lc^{n}_{I,J}(M)$ is artinian.

  \item[\rm{(}iii\rm{)}] If $\lc^{i}_{I,J}(M)=0$ for all
  $i>n$, then $\lc^{n}_{I,J}(M)={\fa}^j\lc^{n}_{I,J}(M)$ for any
  $\fa\in\tilde{W}(I,J)$ and all $j\geq0$.
  Thus $\lc^{n}_{I,J}(M)=0$ if and only if there exist $\fa\in\tilde{W}(I,J)$
  and $m\in\mathbb{N}_0$ such that ${\fa}^m\lc^{n}_{I,J}(M)=0$.

\end{enumerate}
In particular for $n=\emph{dim}M$ and $\emph{cd}(I,J,M)$
\begin{proof}
Apply Corollary \ref{4.5} and \cite[Theorem 3.1]{AghMel2}.
\end{proof}

\end{cor}


Although the following theorem is seemed to be similar to Theorem \ref{4.4}, but it is
more useful and general than \ref{4.4} for finite $R$-modules.

\begin{thm}\label{4.7}
Let $M$ be a finite $R$ module and $n\in\mathbb{N}_0$ be such that $\lc^{i}_{I,J}(M)$
belongs to $\mathcal S$ for all $i> n$. If $\fb$ is an ideal of $R$ such that
 $\lc^{n}_{I,J}(M/{\fb}M)$ belongs to $\mathcal S$,
then the module $\lc^{n}_{I,J}(M)/{\fb}\lc^{n}_{I,J}(M)$ belongs to
 $\mathcal S$.
\end{thm}
\begin{proof}
Let $\fb=(b_1,\dots,b_r)$ and consider the map
$ f : M^r \to M $, defined by
$f(x_1,\dots,x_r)=\sum_1^r b_ix_i$.
Then $\Image f=\fb M$ and $\Coker f=M/\fb M$.
Since $\h^i_{I,J}(M)$ is in $\mathcal S$ for all $i>n$ and
$\Supp(\Ker f)\subseteq\Supp(M)$, it follows from Theorem \ref{4.1}
that $\h^{n+1}_{I,J}(\Ker f)$ is also in $\mathcal S$.
By hypothesis $\h^n_{I,J}(\Coker f)$ belongs to $\mathcal S$.
Hence by \cite[Corollary 3.2]{Mel1} $\Coker{\h^n_{I,J}(f)}$, which equals to
$\h^n_{I,J}(M)/{{\fb}\h^n_{I,J}(M)}$, is in $\mathcal S$.
\end{proof}


\begin{cor}\label{4.8}
Let $M$ be a finite $R$-module and $n\in\mathbb{N}$ be such that
$\lc^{i}_{I,J}(M)\in\mathcal{S}$ for all $i> n$. Then
$\lc^{n}_{I,J}(M)/{\fa}\lc^{n}_{I,J}(M)\in\mathcal S$ for any
$\fa\in\tilde{W}(I,J)$, in particular for $\fa=I$.
\begin{proof}
Let $\fa\in\tilde{W}(I,J)$. Since $M/{\fa}M$ is $\fa$-torsion $R$-module,
thus the assertion follows from Lemma \ref{4.3} (iv), \cite[Corollary 1.13]{TakYoYo}
and Theorem \ref{4.7}.
\end{proof}
\end{cor}


\begin{cor}\label{4.9}
Let $M$ be a finite $R$-module. Let $n\in\mathbb{N}$ be such that
$\lc^{i}_{I,J}(M)\in\mathcal{S}$ for all $i>n$. Then
$\lc^{n}_{I,J}(M)\in\mathcal{S}$ if and only if there exists
$m\in\mathbb{N}_0$ such that $\fa^m\lc^{n}_{I,J}(M)\in\mathcal{S}$.
\end{cor}


\begin{cor}\label{4.10}
Let $M$ be a finite $R$-module. Let $\fa\in\tilde{W}(I,J)$ and
$n\in\mathbb{N}$.
\begin{enumerate}

  \item[\rm{(}i\rm{)}] If $\lc^{i}_{I,J}(M)$ is finite for all $i>n$, then
  $\lc^{n}_{I,J}(M)$ is finite if and only if there exist $\fa\in\tilde{W}(I,J)$
  and $m\in\mathbb{N}_0$ such that $\fa^m\lc^{n}_{I,J}(M)$ is finite
  if and only if there exist $\fa\in\tilde{W}(I,J)$
  and $m\in\mathbb{N}_0$ such that
  $\lc^{n}_{I,J}(M)/(0:_{\lc^{n}_{I,J}(M)}\fa^m)$ is finite

  \item[\rm{(}ii\rm{)}] If $\lc^{i}_{I,J}(M)$ is artinian for
  all $i>n$, then $\lc^{n}_{I,J}(M)$ is artinian if and only if there exists
  $m\in\mathbb{N}_0$ such that ${\fa}^m\lc^{n}_{I,J}(M)$ is artinian.

  \item[\rm{(}iii\rm{)}] If $\lc^{i}_{I,J}(M)=0$ for all
  $i>n$, then $\lc^{n}_{I,J}(M)={\fa}^j\lc^{n}_{I,J}(M)$ for all $j\in\mathbb{N}_0$.
  Thus $\lc^{n}_{I,J}(M)=0$ if and only if there exists $m\in\mathbb{N}_0$ such that
  ${\fa}^m\lc^{n}_{I,J}(M)=0$, in particular for $n=\dim M$ and $\cd(I,J,M)$.

  \item[\rm{(}iv\rm{)}] If $R$ is local and $\lc^{i}_{I,J}(M)$ is finite for all $i>n$,
  then $\lc^{n}_{I,J}(M)={\fa}^j\lc^{n}_{I,J}(M)$ for all $j\geq0$.

\end{enumerate}
\begin{proof}
Apply Corollary \ref{4.9}, \cite[Proposition 1]{LotPay}, and
\cite[Theorem 3.1]{AghMel2}.
\end{proof}
\end{cor}


The following result is more useful whenever $R$ is a local ring and
$I$ is a proper ideal. (see \cite[Lemma 2.1]{Bro})

\begin{cor}\label{4.11}
Let $M$ be a non-zero ZD-module with $d:=\dim M$ and $~t:=\cd(I,J,M)$.
Let $\fa\in\tilde{W}(I,J)$ be such that ${\fa}^m\subseteq\emph{Jac}(R)$ for some
$m\in\mathbb{N}_0$.

\begin{enumerate}

  \item[\rm{(}i\rm{)}] If $t\geq1$, then ${\fa}^j\lc^{t}_{I,J}(M)$
  is not finite for all $j\geq0$.

  \item[\rm{(}ii\rm{)}] If $d\geq1$, then $\lc^{d}_{I,J}(M)$ is finite
  if and only if $~\lc^{d}_{I,J}(M)=0$.

  \item[\rm{(}iii\rm{)}] If $d\geq2$ and $\lc^{d}_{I,J}(M)$ is finite,
  then $\lc^{d-1}_{I,J}(M)/{\fa}^j\lc^{d-1}_{I,J}(M)$ has finite length
  for all $j\geq0$.

\end{enumerate}
In particular when $R$ is local ring and $I\neq R$.
\begin{proof}
(i), (ii) Apply Corollary \ref{4.6} (iii) and Nakayama$^,$s Lemma.\\
(iii) Apply part (ii) and Theorem \ref{4.4} for the class of finite
length.
\end{proof}
\end{cor}


\begin{cor}\label{4.12}
Let $M$ be a non-zero finite $R$-module and set
$t:=\emph{sup}\big\{i\geq1\mid\lc^{i}_{I,J}(M)
\emph{~is not finite}\big\}$,
$n:=\emph{cd}(I,J,M)$, and $r:=\emph{dim}M/JM$.

\begin{enumerate}

  \item[\rm{(}i\rm{)}] If $\fa\in\tilde{W}(I,J)$ is such that
  ${\fa}^m\subseteq\emph{Jac}(R)$ for some $m\in\mathbb{N}_0$, then $n\geq1$
  if and only if $n=t$.

  \item[\rm{(}ii\rm{)}] Let $(R,\fm)$ be a local ring and $r\geq1$
  be an integer such that $\lc^{r}_{I,J}(M)$ is finite. Let
  $\fa\in\tilde{W}(I,J)$ be such that ${\fa}^m+J\subseteq\fm$ for some
  $m\in\mathbb{N}_0$. Then $~\lc^{r}_{I,J}(M)=0$.

  \item[\rm{(}iii\rm{)}] If $(R,\fm)$ is a local ring and $I+J$ is an
  $\fm$-primary ideal, then $r=n=t$.

\end{enumerate}

\begin{proof}
Apply Corollaries \ref{4.9}, \ref{4.10}, and \cite[Theorems 4.3,4.5]{TakYoYo}.
\end{proof}
\end{cor}






\begin{cor}\label{4.13}
Let $M$ be a finite $R$-module and $n\in\mathbb{N}$.
If $\lc^{i}_{I,J}(M)$ is artinian for all $i> n$. Then
$\lc^{n}_{I,J}(M)/{\fa}\lc^{n}_{I,J}(M)$ is artinian for any
$\fa\in\tilde{W}(I,J)$.
\end{cor}
\begin{proof}
Note that $\lc^{n}_{I,J}(M/{\fa}M)= 0$ for any $\fa\in\tilde{W}(I,J)$
and all $n\ge 1$. Now, apply Theorem \ref{4.7} for the class of artinian modules.
\end{proof}


\begin{rem}\label{4.14}
\emph{In \ref{4.13} we have to assume that $n\ge 1$.
Take an ideal $I$ in a ring $R$ such that $R/I$ is not artinian. Let $J=0$ and
 $M=R/I$. Then $\lc^{i}_{I,J}(M)= 0$ for $i\ge 1$, and $\G_{I}(M)= M$.
On the other hand $M/{I}M\cong M$. Thus
  $\lc^{0}_{I,J}(M)/{I}\lc^{0}_{I,J}(M)$ is not artinian}.
\end{rem}


Properties of Serre classes of modules and the previous results motivate us to
introduce the following definition as a generalization of the concept of
cohomological dimension. (see \cite{ChWa}).

\begin{defn}\label{4.15}
\emph{Let $I$ , $J$ be two ideals of $R$ and let $M$ be an $R$-module.
For a Serre subcategory $\mathcal{S}$ of the category of $R$-modules, we define
Serre cohomological dimension of $M$ with respect to $(I,J)$, by
$${\textmd{cd}}_{\mathcal{S}}(I,J,M)={\textmd{sup}}\big\{{i\in\mathbb{N}_0}\mid
\lc^{i}_{I,J}(M)\not\in\mathcal{S}\big\},$$
if this supremum exists, and $\infty$ otherwise.\\
It is easy to see that
${\textmd{cd}}_{\mathcal{S}}(I,J,M)={\textmd{inf}}\big\{{n\in\mathbb{N}_0}\mid
\lc^{i}_{I,J}(M)\in\mathcal{S} ~\textmd{~for ~all} ~~i>n\big\}$}.
\end{defn}


\begin{rem} \label{4.16}
\emph{For an arbitrary Serre class $\mathcal{S}$ , we have
${\textmd{cd}}_{\mathcal{S}}(I,J,M)\leq{\emph{cd}}(I,J,M)$, and if
in Definition 4.16, we let $\mathcal{S}:=0$ then we have
$${\textmd{cd}}_{\mathcal{S}}(I,J,M)={\textmd{sup}}\big\{{i\in\mathbb{N}_0}\mid
\lc^{i}_{I,J}(M)\neq0\big\}=\textmd{cd}(I,J,M).$$
Also, if $\mathcal{S}$ is the class of Artinian $R$-modules, we get
$${\textmd{cd}}_{\mathcal{S}}(I,J,M)={\textmd{sup}}\big\{{i\in\mathbb{N}_0}\mid
\lc^{i}_{I,J}(M)~ {\textmd{is not Artinian $R$-module}}\big\}.$$
We denote it by $\q(I,J,M)$}.
\end{rem}


\begin{prop} \label{4.17}
Let $M$ be a ZD-module of finite krull dimension or finite $R$-module. Let
$\mathcal{S}$ be a Serre class and $n:=\emph{cd}_{\mathcal{S}}(I,J,M)\geq1$.
Then ${\fa}^j\lc^{n}_{I,J}(M)\not\in\mathcal{S}$ for any
$\fa\in\tilde{W}(I,J)$ and all $j\geq0$,
in particular ${\fa}^j\lc^{n}_{I,J}(M)\neq0$
for any $\fa\in\tilde{W}(I,J)$ and all $j\geq0$.
\begin{proof}
Apply Corollaries \ref{4.5} , \ref{4.9}.
\end{proof}
\end{prop}


It is well known that if $\fa$ is an ideal of $R$ and $M$ , $N$ are finite $R$-modules
with $\textmd{Supp}(N)\subseteq \textmd{Supp}(M)$, then
${\textmd{cd}}(\fa,N)\leq {\textmd{cd}}(\fa,M)$. The next result is a generalization
of this fact.

\begin{prop} \label{4.18}
Let $M$ , $N$ be finite $R$-modules such that $\emph{Supp}(N)\subseteq \emph{Supp}(M)$.
Then ${\emph{cd}}_{\mathcal{S}}(I,J,N)\leq {\emph{cd}}_{\mathcal{S}}(I,J,M)$.
\end{prop}
\begin{proof}
Apply Corollary \ref{4.2}.
\end{proof}


\begin{cor} \label{4.19}
Let $M$ , $N$ be finite $R$-modules such that $\emph{Supp}(N)\subseteq \emph{Supp}(M)$.
Then  ${\emph{cd}}(I,J,N)\leq {\emph{cd}}(I,J,M)$.
\end{cor}


\begin{cor} \label{4.20}
For a finite $R$-module $M$, there exist the following equalities.
\begin{align*}
\emph{cd}_{\mathcal{S}}(I,J,M)&={\emph{max}~}\big\{\emph{cd}_
{\mathcal{S}}(I,J,R/\mathfrak{p})\mid\mathfrak{p}\in\ {\emph{Ass}}(M)\big\}\\
&={\emph{max}~}\big\{\emph{cd}_{\mathcal{S}}(I,J,R/\mathfrak{p})\mid
\mathfrak{p}\in\ {\emph{Min Ass}}(M)\big\}\\ &={\emph{max}~}\big\{\emph{cd}_
{\mathcal{S}}(I,J,R/\mathfrak{p})\mid\mathfrak{p}\in\ {\emph{Supp}}(M)\big\}\\
&={\emph{max}~}\big\{\emph{cd}_{\mathcal{S}}(I,J,R/\mathfrak{p})\mid
\mathfrak{p}\in\ {\emph{Min Supp}}(M)\big\}\\ &={\emph{max}~}\big\{
\emph{cd}_{\mathcal{S}}(I,J,N)\mid N ~\emph{is a finite submodule } M \big\}\\
&={\emph{max}~}\big\{i\geq0\mid H^i_{I,J}(R/\mathfrak{p})\not\in\mathcal{S},
~\emph{for some} ~\mathfrak{p}\in{\emph{Ass}}(M)\big\}\\
&=\emph{min~}\big\{{n\geq0}\mid\lc^{i}_{I,J}(R/\fp)\in\mathcal{S}
~\emph{~for ~all} ~~i>n ~\emph{and  ~all} ~~i>n\big\}.\end{align*}
\end{cor}

\begin{proof}
Apply Corollary \ref{4.1} and Proposition \ref{4.19}.
\end{proof}


\begin{cor} \label{4.21}
Let $M$ , $N$ be finite $R$-modules such that
$\emph{Supp}(N)\subseteq \emph{Supp}(M)$.
Then  $\q(I,J,N)\leq \q(I,J,M)$.
\end{cor}


\begin{prop} \label{4.22}
Let $0\rightarrow L\rightarrow M\rightarrow N\rightarrow0$
be an exact sequence of finite $R$-modules. Then
$\emph{cd}_{\mathcal{S}}(I,J,M)={\emph{max}}~\big\{\emph{cd}_
{\mathcal{S}}(I,J,L)~,~\emph{cd}_{\mathcal{S}}(I,J,N)\big\}$
\end{prop}
\begin{proof}
Let $t:=\textmd{cd}_{\mathcal{S}}(I,J,M)$ and $s:={\textmd{max}~}\big\{
\emph{cd}_{\mathcal{S}}(I,J,L)~,~\emph{cd}_{\mathcal{S}}(I,J,N)\big\}$.
By, Corollary \ref{4.2} , we have $t\geq s$. Now, let $t>s$. Then by
the following exact sequence
$$\cdots\rightarrow H^t_{I,J}(L)\rightarrow H^t_{I,J}(M)\rightarrow
H^t_{I,J}(N)\rightarrow\cdots,$$
we get $H^t_{I,J}(M)\in\mathcal{S}$
which is a contradiction with $\textmd{cd}_{\mathcal{S}}(I,J,M)=t$.
\end{proof}


\begin{cor} \label{4.23}
For a noetherian ring $R$ there exists the following equality.
$$\emph{cd}_{\mathcal{S}}(I,J,R)={\emph{sup}~}\big\{\emph{cd}_
{\mathcal{S}}(I,J,N)\mid ~N~ {\textmd{is a finite $R$-module}} \big\}.$$
In particular, for $r\in\mathbb{N}_0$ the following statements are equivalent:
\begin{enumerate}

\item[(i)] $H^j_{I,J}(R)\in\mathcal{S}$ for all $j>r$.

\item[(ii)] $H^j_{I,J}(M)\in\mathcal{S}$ for all $j>r$ and all finite $R$-module $M$.
\end{enumerate}
\end{cor}


\bibliographystyle{amsplain}

\end{document}